\documentclass[11pt]{article}
\usepackage{amsmath,amsthm,amssymb,array}
\usepackage{macros}
\usepackage[dvips]{epsfig}
\usepackage{enumerate,pdfsync,psfrag}

\newcommand{\sa}{s_\varphi}
\newcommand{\ca}{c_\varphi}

\begin{document}

\title{Control of a Bicycle Using Virtual Holonomic Constraints}
\author{Luca Consolini, Manfredi Maggiore,%
\thanks{L. Consolini is with the Department of
  Information Engineering, University of Parma, Viale Usberti
  181/A, Parma, 43124 Italy. E-mail: {\tt lucac@ce.unipr.it}. M. Maggiore is with the Department of Electrical and
  Computer Engineering, University of Toronto, 10 King's College Road,
  Toronto, Ontario, M5S~3G4,
  Canada. E-mail: {\tt maggiore@control.utoronto.ca}.}}

\maketitle

\begin{abstract}
The paper studies the problem of making Getz's bicycle model traverse
a strictly convex Jordan curve with bounded roll angle and bounded
speed. The approach to solving this problem is based on the virtual
holonomic constraint (VHC) method. Specifically, a VHC is enforced
making the roll angle of the bicycle become a function of the
bicycle's position along the curve. It is shown that the VHC can be
automatically generated as a periodic solution of a scalar periodic
differential equation, which we call virtual constraint generator.
Finally, it is shown that if the curve is sufficiently long as
compared to the height of the bicycle's centre of mass and its wheel
base, then the enforcement of a suitable VHC makes the bicycle
traverse the curve with a steady-state speed profile which is periodic
and independent of initial conditions. An outcome of this work is a
proof that the constrained dynamics of a Lagrangian control system
subject to a VHC are generally not Lagrangian.
\end{abstract}

\section{Introduction}
This paper investigates the problem of maneuvering a bicycle along a
closed Jordan curve $\C$ in the horizontal plane in such a way that
the bicycle does not fall over and its velocity is bounded.  The
simplified bicycle model we use in this paper, developed by Neil
Getz~\cite{Get94,GetMar95}, views the bicycle as a point mass with a
side slip velocity constraint, and models its roll dynamics as those
of an inverted pendulum, see Figure~\ref{fig:bicycle}. The model
neglects, among other things, the steering kinematics and the wheels
dynamics with the associated gyroscopic effect. %% The motion of Getz's
%% bicycle when the contact point of the rear wheel is made to follow the
%% curve $\C$ is described by a Lagrangian control
%% system~\cite{BloBaiCroMar03}, i.e., a system of the form
%% %
%% %
%% %
%% \begin{equation}\label{eq:lagrangian}
%% \frac{d}{dt} \frac{\partial L}{\partial \dot q} - \frac{\partial
%%   L}{\partial q} = G u,
%% \end{equation}
%% %
%% %
%% %
%% with control input $u$.

In~\cite{HauSacFre04}, Hauser-Saccon-Frezza investigate the
maneuvering problem for Getz's bicycle using a dynamic inversion
approach to determine bounded roll trajectories.  They constrain the
bicycle on the curve and, given a desired velocity signal $v(t)$, they
find a trajectory with the property that the velocity of the bicycle
is $v(t)$ and its roll angle $\varphi$ is in the interval
$(-\pi/2,\pi/2)$, i.e, the bicycle doesn't fall
over. In~\cite{HauSac06}, Hauser-Saccon develop an algorithm to
compute the minimum-time speed profile for a point-mass motorcycle
compatible with the constraint that the lateral and longitudinal
accelerations do not make the tires slip, and apply their algorithm to
Getz's bicycle model.

The problem of maneuvering Getz's bicycle along a closed curve is
equivalent to moving the pivot point of an inverted pendulum around
the curve without making the pendulum fall over. On the other hand,
the seemingly different problem of maneuvering Hauser's PVTOL
aircraft~\cite{HauSasMey92} along a closed curve in the vertical plane
can be viewed as the problem of moving the pivot of an inverted
pendulum around the curve without making the pendulum fall over. The
two problems are, therefore, closely related, the main difference
being the fact that in the former case the pendulum lies on a plane
which is orthogonal to the plane of the curve, while in the latter
case it lies on the same plane. In~\cite{ConMagNieTos10}, the path
following problem for the PVTOL was solved by enforcing a virtual
holonomic constraint (VHC) which specifies the roll angle of the PVTOL
as a function of its position on the curve.  In this paper we follow a
similar approach for the bicycle model, and impose a VHC relating the
bicycle's roll angle to its position along the curve. However, rather
than finding one VHC, as we did in~\cite{ConMagNieTos10}, we show how
to automatically generate VHCs as periodic solutions of a scalar
periodic differential equation which we call the {\em virtual
  constraint generator}. We show that if the path is sufficiently long
compared to the height of the bicycle's centre of mass and the wheel
base, then the VHC can be chosen so that on the constraint manifold
the bicycle traverses the entire curve with bounded speed, and its
speed profile is periodic in steady-state.  Finally, we design a
controller that enforces the VHC, and recovers the asymptotic
properties of the bicycle on the constraint manifold.

\begin{figure}[bht]
\psfrag{a}{$\varphi$}
\psfrag{b}{$\alpha$}
\psfrag{p}{$\psi$}
\psfrag{l}{$b$}
\psfrag{k}{$p$}
\psfrag{h}{$h$}
\psfrag{m}{$m$}
\psfrag{c}{$(x,y)$}

\centerline{\includegraphics[width=.5\textwidth]{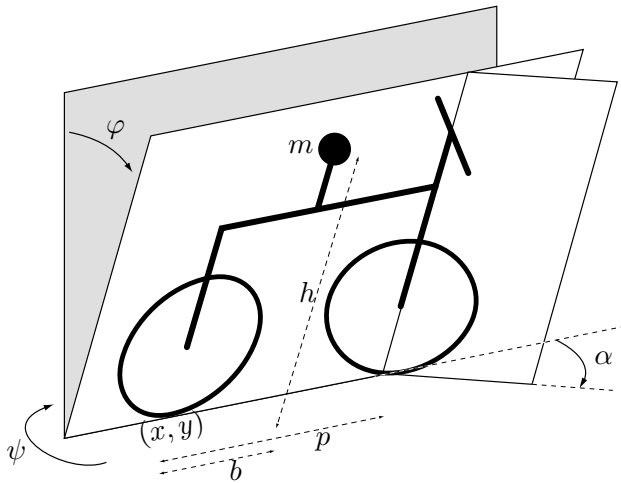}}
\caption{Getz's bicycle model.}
\label{fig:bicycle}
\end{figure}

The concept of VHC is a promising paradigm for motion control. It is one
of the central ideas in the work of Grizzle and collaborators on biped
locomotion (e.g.,~\cite{PleGriWesAbb03} and~\cite{WesGriKod03}), where
VHCs are used to encode different walking gaits.  The work of Shiriaev
and collaborators in~\cite{Can04,ShiPerWit05,FreRobShiJoh08}
investigates VHCs for Lagrangian control systems, i.e., systems of the
form~\cite{Blo03}
\begin{equation}\label{eq:lagrangian}
\frac{d}{dt} \frac{\partial L}{\partial \dot q} - \frac{\partial
  L}{\partial q} = G u,
\end{equation}
with control input $u$ and smooth Lagrangian $L(q,\dot q) = (1/2) \dot
q\trans D(q) \dot q - V(q)$, with $D = D\trans$ positive definite.
In~\cite{ShiPerWit05}, the authors consider systems of the
form~\eqref{eq:lagrangian} with degree of underactuation one. They
find an integral of motion for the constrained dynamics, and use it to
select a desired closed orbit on the constraint manifold. This orbit
is then stabilized by linearizing the control system along it, and
designing a time-varying controller for the
linearization. In~\cite{ShiFreGus10}, these ideas are extended to
systems with degree of underactuation greater than one, and
in~\cite{ShiFreRobJohSan07} they are applied to the stabilization of
oscillations in the Furuta pendulum.  In~\cite{MagCon13}, we
investigated VHCs for Lagrangian control systems with degree of
underactuation one. We introduced and characterized a notion of
regularity of VHCs, and we presented sufficient conditions under which
the reduced dynamics on the constraint manifold (described by a
second-order unforced system) are Lagrangian (i.e., they satisfy the
Euler-Lagrange equations, which have the form~\eqref{eq:lagrangian}
with zero right-hand side). An outcome of this paper
(Proposition~\ref{prop:internal_dynamics_general_properties}) is a
simple sufficient condition under which the reduced dynamics are {\em
  not} Lagrangian. We refer the reader to
Remarks~\ref{rem:shiriaev},~\ref{rem:Lagrange}, and~\ref{rem:not_EL}
for more details.

This paper is organized as follows. In Section~\ref{sec:problem} we
present Getz's bicycle model and we formulate the maneuvering problem
investigated in this paper. In Section~\ref{sec:vcg} we present the
virtual constraint generator idea. The main result is
Proposition~\ref{prop:generator:periodic_solutions} which gives a
constructive methodology to find VHCs for Getz's bicycle that meet the
requirements of the maneuvering problem. In
Section~\ref{sec:constrained_motion} we analyze the motion of Getz's
bicycle on the virtual constraint manifold. In
Proposition~\ref{prop:internal_dynamics_general_properties} we provide
a general result with sufficient conditions under which an unforced
second-order system of a certain form possesses an asymptotically
stable closed orbit. In
Proposition~\ref{prop:internal_dynamics_bicycle} we apply this
result to Getz's bicycle model. In Section~\ref{sec:solution} we bring
these results together and solve the maneuvering
problem. Finally, in Section~\ref{sec:numerics} we make remarks on
numerical implementation of the proposed controller.

{\bf Notation.} Throughout this paper we use the following
notation. If $x$ is a real number and $T>0$, the number $x$ modulo $T$
is denoted by $[x]_T$. We let $[\Re]_T : = \{[x]_T: x \in \Re\}$. The
set $[\Re]_T$ is diffeomorphic to the unit circle. We let $\pi: \Re
\to [\Re]_T$ be defined as $\pi(t) = [t]_T$. Then, $\pi$ is a smooth
covering map (see~\cite[p.91]{Lee13}). If $M$ is a smooth manifold,
and $h: [\Re]_T \to M$ is a smooth function, we define $\tilde h:=h
\circ \pi : \Re \to M$. This is a $T$-periodic function. Moreover,
by~\cite[Theorem~4.29]{Lee13}, $\tilde h$ is smooth if and only if
$h$ is smooth. Finally, $\image(h)$ denotes the image of a function
$h$.

\section{Problem formulation}\label{sec:problem}

Consider the bicycle model depicted in Figure~\ref{fig:bicycle}, with
the following variable conventions (taken from~\cite{HauSacFre04}):
\begin{itemize}

\item $(x,y)$ - coordinates of the point of contact of the rear wheel

\item $\varphi$ - roll angle (a positive $\dot \varphi$ implies that
  the bicycle leans to the right)

\item $\psi$ - yaw angle (a positive $\dot \psi$ means that the
  bicycle turns right)

\item $\alpha$ - projected steering angle on the $(x,y)$ plane

\item $b$ -  distance between the projection of the centre of mass and
  the point of contact of the rear wheel

\item $p$ -  wheel base

\item $h$ - pendulum length

\item $v$ - forward linear velocity of the bicycle

\item $f$ - thrust force.
\end{itemize}
We denote $\bar \kappa = (\tan \alpha) / p = \dot \psi / v$. For a
given velocity signal $v(t)$ and steering angle signal $\alpha(t)$,
$\bar \kappa(t)$ represents the curvature of the path $(x(t),y(t))$
traced by the point of contact of the rear wheel.  In~\cite{GetMar95},
the bicycle of Figure~\ref{fig:bicycle} was modelled by writing the
Lagrangian of the unconstrained bicycle, incorporating the
nonholonomic constraints that the wheels roll without slipping in the
Lagrangian, and then extracting the model through the
Lagrange-d'Alembert equations as in~\cite[Section 5.2]{Blo03}. The
resulting model, which we'll refer to as {\em Getz's bicycle model},
reads as
\begin{align}
&\dot {\bar \kappa} = \tau \nonumber\\
&M \begin{bmatrix} \ddot \varphi \\ \dot v 
\end{bmatrix} = 
F + B  \begin{bmatrix}
  \tau \\ f
\end{bmatrix}, \label{eq:original_system}
\end{align}
where $\tau$, the time derivative of the {curvature $\bar
  \kappa(t)$}, and $f$ are the control inputs and, denoting $\sa =
\sin \varphi$, $\ca=\cos \varphi$,
\[
\begin{aligned}
& M = \begin{bmatrix} h^2 & b h \ca \bar \kappa \\
b h \ca \bar \kappa & 1 + (b^2+h^2 \sa^2) \bar \kappa^2 - 2 h {\bar \kappa} \sa
\end{bmatrix}, \\
& F= \begin{bmatrix} g h \sa - (1 - h
   {\bar \kappa} \sa) h \ca {\bar \kappa} v^2 \\ (1 - h {\bar \kappa} \sa) 2 h \ca {\bar \kappa}
   v \dot \varphi + b h {\bar \kappa} \sa \dot \varphi^2
 \end{bmatrix}, \\
& B =\begin{bmatrix} -b h \ca v & 0 \\ -(b^2 {\bar \kappa} -
  h \sa(1-h {\bar \kappa} \sa)) v & 1/m
  \end{bmatrix}.
\end{aligned}
\]
In this model, $M$ is the inertia matrix, $F$ represents the sum of
Coriolis, centrifugal and conservative forces, and $B$ is the input
matrix. Now consider a $C^3$ closed Jordan curve $\C$ in the $(x,y)$
plane with regular parametrization $\sigma: [\Re]_T \to \Re^2$, not
necessarily unit speed. Let $\kappa:[\Re]_T \to \Re$ denote the signed
curvature of $\C$. Throughout this paper, we assume the following.
\begin{assumption}
\label{assump_convex}
The curve $\C$ is strictly convex, i.e., $\kappa(s) > 0$ for all $s
\in [\Re]_T$.
\end{assumption}
In this paper we investigate the dynamics of the bicycle when the
point $(x,y)$ is made to move along the curve $\C$ by an appropriate
choice of the steering angle. In order to derive the constrained
dynamics, suppose that $( x(0), y(0) ) \in \C$, i.e., $(x(0),y(0)) =
\sigma(s_0)$, for some $s_0 \in [\Re]_T$.  A point $\sigma(s(t))$
moving on $\C$ has linear velocity
\begin{equation}
\label{eqn_for_v}
v(t) = \|\sigma'(s(t))\| \dot s(t) 
\end{equation}
and acceleration
\begin{equation}
\label{eqn_for_v_dot}
\dot v(t) = \|\sigma'(s(t))\| \ddot s(t) + \frac{\dot s^2(t)}{
 \|\sigma'(s(t))\|} \sigma'(s(t))\trans \sigma''(s(t)).
\end{equation}

For an arbitrary input signal $f(t)$, 
$(x(t), y(t))$ traverses $\C$
with velocity $v(t)$ if and only if $(x(0),y(0)) \in \C$, $(\dot
x(0),\dot y(0))$ is tangent to $\C$, $\bar \kappa(0) = \kappa(s_0)$,
and the input signal $\tau(t)$ is
chosen to be $\tau(t) = \kappa'(s(t)) \dot s(t)$, where $\kappa'(s(t)) =
\frac{d\kappa} {ds} (s(t))$. With this choice, we obtain 
\begin{equation}
\label{eqn_for_bar_kappa}
{\bar \kappa}(t) = \kappa(s(t)),
\end{equation}
where $s(t)$ and $\dot s(t)$ are solutions of a differential equation
to be specified later.  The motion of the bicycle on the curve $\C$ is
now found by
substituting~(\ref{eqn_for_v}),~(\ref{eqn_for_v_dot}),~(\ref{eqn_for_bar_kappa})
and $\tau = \kappa'(s) \dot s$ in~(\ref{eq:original_system}):
\begin{equation}\label{eq:constrained_system}
\bar{M} \begin{bmatrix} \ddot \varphi \\ \|\sigma^\prime\| \ddot {s} +
\frac{(\sigma')\trans \sigma''}{\|\sigma'\|} \dot{s}^2 
\end{bmatrix} = 
\bar F + \begin{bmatrix} 0 \\1/m \end{bmatrix} f\;,
\end{equation}
where $\bar M=M|_{\bar \kappa =\kappa(s)}$ and
\[
\bar F = \left(F+B \begin{bmatrix} 1 \\ 0
\end{bmatrix}\kappa'(s) \dot s
\right)\bigg|_{\bar \kappa =\kappa(s), v=\|\sigma'(s)\|\dot s}.
\]
Note that we have used the control input $\tau$ to make the curve $\C$
invariant, so in~\eqref{eq:constrained_system} we are left with one
control input, the thrust force $f$. One can check that
system~\eqref{eq:constrained_system} is a Lagrangian control system
with input $f$, i.e., it has the form~\eqref{eq:lagrangian} with $q=
(\varphi,s)\in \cQ =S^1 \times [\Re]_T$, $G = [0 \ \ 1/m]\trans$, and
\[
L(q,\dot q) = \frac 1 2  [\dot
  \varphi \ \ \|\sigma^\prime(s)\| \dot{s}] \bar{M}(q) \begin{bmatrix}
  \dot \varphi \\ \|\sigma^\prime(s)\| \dot{s}
\end{bmatrix}- g h \cos \varphi.
\]
Since the control force $f$ enters nonsingularly in the $\ddot s$
equation, we can define a feedback transformation for $f$
in~\eqref{eq:constrained_system} such that $\ddot s=u$, where $u$ is
the new control input. With this transformation, the motion of the
bicycle when its rear wheel is made by feedback control to follow $\C$
reads as
\begin{equation}\label{eq:sys}
\begin{aligned}
& \ddot \varphi = h^{-1}g \sa - h^{-1}\Big[ (1 - h \kappa(s)
    \sa)\kappa(s) \|\sigma'(s)\| + b \kappa'(s) \\
&  +\frac{b \kappa(s)}{\|\sigma'(s)\|^2} {\sigma'(s)}\trans
   \sigma''(s)\Big]\ca \|\sigma'(s)\| \dot s^2 - a(s) \ca u \\
& \ddot s = u,
\end{aligned}
\end{equation}
where $a(s) = b h^{-1} \kappa(s) \|\sigma'(s)\|$. In the above
equation, $u$ is the new control input, and it represents the
acceleration of the curve parameter $s$. We will denote $\cX := \{
(q,\dot q) \in S^1 \times [\Re]_T \times \Re^2\}$ the state space
of~\eqref{eq:sys}.

\begin{remark}
{\rm If $\sigma(s)$ is a unit speed parametrization of $\C$ (i.e. it
  satisfies $\|\sigma'(s)\|=1$), then the first differential equation
  in~\eqref{eq:sys} reduces to
$h \ddot \varphi = g \sa - \Big[ (1 - h \kappa(s) \sa)\kappa(s) + b
 \kappa'(s)\Big]\ca \dot s^2 - b \kappa(s) \ca u$.
}
\end{remark}
\begin{problem}{Maneuvering Problem}
Find a feedback $u(q,\dot q)$ for system~\eqref{eq:sys} such that
there exists a set of initial conditions $\Omega$ with the property
that, for all $( q(0),\dot q(0) ) \in \Omega$, the bicycle does not
overturn, i.e., $|\varphi(t)| < \pi/2$ for all $t \geq 0$, and
traverses the entire curve $\C$ in one direction, i.e., there exists
$\bar t>0$ such that $| \dot s(t) | >0$ for all $ t \geq \bar t$. Moreover, the
speed $\dot s(t)$ of the bicycle on $\C$ should remain bounded.
\end{problem}
Our solution of this problem relies on the notion of VHC.
\begin{definition}[\cite{MagCon13}]
A {\bf virtual holonomic constraint (VHC)} for system~\eqref{eq:sys}
is a relation $\varphi = \Phi(s)$, where $\Phi: [\Re]_T \to S^1$ is
smooth and the set
$\Gamma = \{(q,\dot q)\in \cX: \varphi = \Phi(s), \ \dot \varphi = \Phi'(s)
\dot s\}$
is controlled invariant. That is, there exist a smooth feedback
$u(q,\dot q)$ such that $\Gamma$ is invariant for the closed-loop
system. The set $\Gamma$ is called the {\bf constraint manifold}
associated with the VHC $\varphi = \Phi(s)$.
\end{definition}
The definition above formalizes the notion of VHC used
in~\cite{WesGriKod03} in the context of biped locomotion.  The
constraint manifold $\Gamma$ is a two-dimensional submanifold of $\cX$
parametrized by $(s,\dot s)$, and therefore diffeomorphic to the
cylinder $[\Re]_T \times \Re$.  It is the collection of all those
phase curves of~\eqref{eq:sys} that can be made to satisfy the
constraint $\varphi = \Phi(s)$ via feedback control.  In order to
solve the maneuvering problem, we look for VHCs $\varphi = \Phi(s)$
such that $| \Phi (s) | < \pi/2$ for all $s \in [\Re]_T$, and then
stabilize the associated virtual constraint
manifold. In~\cite{HauSacFre04}, Hauser-Saccon-Frezza find ``bounded
roll trajectories,'' i.e., controlled trajectories of~\eqref{eq:sys}
along which the roll angle $\varphi$ is bounded in the interval
$(-\pi/2,\pi/2)$.  In our context, each VHC $\varphi = \Phi(s)$
provides a {\em family} of bounded roll trajectories. Once $\Gamma$
has been made invariant via feedback control, bounded roll
trajectories can be obtained by picking arbitrary $(s_0,\dot s_0) \in
[\Re]_T \times \Re$, and picking as initial condition
in~\eqref{eq:sys}, $(\varphi(0),s(0)) = (\Phi(s_0),s_0)$, $(\dot
\varphi(0),\dot s(0)) = (\Phi'(s_0)\dot s_0,\dot s_0)$. The resulting
solution $(q(t),\dot q(t))$ will satisfy $\varphi(t) = \Phi(s(t))$,
implying that the roll angle trajectory $\varphi(t)$ is bounded in the
interval $(-\pi/2,\pi/2)$.
\section{The virtual constraint generator}\label{sec:vcg}
In this section we show that VHCs for~\eqref{eq:sys} can be generated
as solutions of a first-order differential equation, which we call the
{\em VHC generator.} This idea was first presented in our previous
work~\cite{ConMag10_1}.  We begin with a sufficient condition for a
relation $\Phi$ to be a VHC for~\eqref{eq:sys}.
\begin{lemma}\label{lem:feasibility}
A relation $\varphi = \Phi(s)$, where $\Phi:[\Re]_T \to S^1$ is
smooth, is a VHC for system~\eqref{eq:sys} if
\begin{equation}\label{eq:feasibility}
(\forall s \in [\Re]_T) \ \Phi'(s) + a(s) \cos \Phi(s) \neq 0.
\end{equation}
\end{lemma}
\begin{proof}
Letting $H(\varphi,s):=\varphi - \Phi(s)$,
condition~\eqref{eq:feasibility} is the requirement that
system~\eqref{eq:sys} with output $H(\varphi,s)$ has a well-defined
uniform relative degree 2 on the set $\{(q,\dot q)\in \cX: \varphi =
\Phi(s)\}$. The associated zero dynamics manifold is precisely
$\Gamma$, and it is controlled invariant.
\end{proof}
The foregoing lemma inspires the following observation. Instead of
guessing a relation $\varphi = \Phi(s)$ and checking whether it is a
VHC for~\eqref{eq:sys}, as in Lemma~\ref{lem:feasibility}, we will
assign a nonzero right-hand side to~\eqref{eq:feasibility}, view the
resulting identity as an ODE, and {\em generate} VHCs by solving the
ODE. More precisely, recall the notation $\tilde h:=h \circ \pi$, and
consider the scalar $T$-periodic differential equation
\begin{equation} \label{eq:constraint_generator}
\dot x = - \tilde a(t)\cos x + \tilde \delta(t),
\end{equation}
where $\tilde \delta:\Re \to \Re \backslash \{0\}$ is a $T$-periodic
function to be assigned. Then, $T$-periodic solutions
of~\eqref{eq:constraint_generator} give rise to VHCs.
\begin{lemma}\label{lem:vcg}
Suppose that $x(t)$ is a $T$-periodic solution
of~\eqref{eq:constraint_generator}, where $\tilde \delta :\Re \to \Re
\backslash \{0\}$ is smooth and $T$-periodic. Then, the relation
$\varphi = \Phi(s)$, with $\Phi:=x \circ \pi^{-1}:[\Re]_T \to \Re$, is
a VHC for~\eqref{eq:sys}.
\end{lemma}
\begin{proof}
Let $x(t)$ be a $T$-periodic solution
of~\eqref{eq:constraint_generator}. Then, $x(t)$ is smooth because the
right-hand side of~\eqref{eq:constraint_generator} is smooth. Since
$x(t)$ is $T$-periodic, $\Phi:=x \circ \pi^{-1}$ is a well-defined
function $[\Re]_T \to \Re$, and by~\cite[Theorem~4.29]{Lee13} it
is smooth. Since $x = \Phi \circ \pi$, we have $\dot x(t) =
\Phi'(\pi(t)) \dot \pi(t) = \Phi'(\pi(t))$, so that $\Phi'(s) =\dot
x(t)|_{t = \pi^{-1}(s)}$. Using~\eqref{eq:constraint_generator}, we get
\begin{align}
\Phi'(s) & = -\tilde a(\pi^{-1}(s)) \cos \Phi(s) + \tilde
\delta(\pi^{-1}(s)) \nonumber \\
&=- a(s) \cos \Phi(s) +
\delta(s). \label{eq:Phiprime}
\end{align}
Since $\delta(s) \neq 0$ for all $s \in [\Re]_T$, by
Lemma~\ref{lem:feasibility} the relation $\varphi = \Phi(s)$ is a VHC
for system~\eqref{eq:sys}.  \end{proof}
In light of Lemma~\ref{lem:vcg}, we call the differential
equation~\eqref{eq:constraint_generator} a {\em VHC generator}, for
which one is to pick a $T$-periodic $\tilde \delta:\Re \to \Re
\backslash \{0\}$ yielding a $T$-periodic solution. The next
proposition shows how to pick $\tilde \delta$.
\begin{proposition}\label{prop:generator:periodic_solutions}
Consider the smooth and $T$-periodic differential
equation~\eqref{eq:constraint_generator}, and set $\tilde \delta(t) =
\epsilon \tilde \mu(t)$, where $\tilde \mu: \Re \to (0,\infty)$ is
smooth and $T$-periodic. Let $K^+ = \max_{t \in [0, T]}[ \tilde
  \mu(t)/\tilde a(t)]$, $K^- = \min_{t \in [0,T]} [\tilde
  \mu(t)/\tilde a(t)]$.  Then, for any $x_0 \in (0,\pi/2)$ and $t_0
\in \Re$, the following two properties hold:
\begin{enumerate}[(i)]
\item there exists a unique $\epsilon \in [\epsilon^-,\epsilon^+] = [
  ( c_{x_0}) / K^+,$ $( c_{x_0}) / K^- ]$, such that the solution
  $x(t)$ of~\eqref{eq:constraint_generator} with initial condition
  $x(t_0) = x_0$ is $T$-periodic, and setting $\Phi=x \circ \pi^{-1}$,
  the relation $\varphi = \Phi(s)$ is a VHC for~\eqref{eq:sys}.
\item If $\tilde\mu:\Re \to (0,\infty)$ is chosen so that $K^+ / K^- <
  (\cos x_0)^{-1}$, then the VHC $\varphi=\Phi(s)$ in part (i)
  satisfies $\image(\Phi) \subset (\Phi^-,\Phi^+) \subset (0,\pi/2)$,
  where $\Phi^-=\cos^{-1}\big(\frac{K^+}{K^-} c_{x_0}\big)$ and
  $\Phi^+=\cos^{-1}\big(\frac{K^-}{K^+} c_{x_0}\big)$.
\end{enumerate}
\end{proposition}
\begin{remark}\label{rem:mu}
{\rm By choosing $\tilde \mu(t) \equiv \tilde a(t)$, we have $K^+ = K^- =1$,
  $\epsilon^+ = \epsilon^- = c_{x_0}$.
  Proposition~\ref{prop:generator:periodic_solutions} implies that,
  for all $x_0 \in (0,\pi/2)$, setting $\tilde \delta(t) = \cos (x_0)
  \tilde a(t)$, the VHC generator produces a $T$-periodic solution $x(t)$
  with image contained in $(0,\pi/2)$. As a matter of fact, the
  solution in question is $x(t) \equiv x_0$, so that the relation
  $\varphi =x_0$ is a VHC for~\eqref{eq:sys}.  The bicycle subject to
  this VHC has a constant roll angle as it travels around $\C$.  The
  proposition provides great flexibility in finding VHCs with the
  property that the roll angle is confined within the interval
  $(0,\pi/2)$. All such constraints are compatible with the
  maneuvering problem.}
\end{remark}
\begin{proof}
Pick an arbitrary $x_0 \in (0,\pi/2)$ and $t_0 \in \Re$, and
consider the VHC generator
\begin{equation}\label{eq:constraint_generator:proof}
\dot x = - \tilde a(t) \cos x + \epsilon \tilde \mu(t),
\end{equation}
where $\tilde \mu(\cdot)>0$. Let $\cS(\epsilon)$ denote the solution
of~\eqref{eq:constraint_generator:proof} at time $t_0+T$ with initial
condition $x(t_0) = x_0$. Since the right-hand side
of~\eqref{eq:constraint_generator:proof} is $T$-periodic, the solution
of~\eqref{eq:constraint_generator:proof} with initial condition
$x(t_0) =x_0$ is $T$-periodic if and only if $\cS(\epsilon) = x_0$.
For all $\epsilon > \epsilon^+$ we have
\[
\begin{aligned}
\dot x & > -\tilde a(t) \cos x + \tilde \mu(t) \cos(x_0) / K^-\\
 & > -\tilde a(t) \left[ \cos x - \cos x_0 (\tilde \mu(t) / \tilde
  a(t) ) / K^- \right]] \\
& > -\tilde a(t) (\cos x - \cos x_0).
\end{aligned}
\]
The solution of $\dot x = -\tilde a(t) (\cos x - \cos x_0)$ with
initial condition $x(t_0) = x_0$ is $x(t) \equiv x_0$. Therefore, by
the comparison lemma (\cite[Lemma 3.4]{Kha02}), for all $\epsilon>
\epsilon^+$ the solution of~\eqref{eq:constraint_generator:proof} with
initial condition $x(t_0) = x_0$ satisfies $x(t)> x_0$ for all $t >
t_0$, so that $\cS(\epsilon)> x_0$ for all $\epsilon> \epsilon^+$.  A
similar argument can be used to show that for all $\epsilon<
\epsilon^-$ the solution of~\eqref{eq:constraint_generator:proof}
satisfies $x(t) < x_0$ for all $t>t_0$, so that $\cS(\epsilon)< x_0$
for all $\epsilon<\epsilon^-$. By continuity of solutions with respect
to parameters, the map $\cS$ is continuous, and therefore there exists
$\epsilon \in [\epsilon^-,\epsilon^+]$ such that $\cS(\epsilon) =
x_0$. The corresponding solution $x(t)$ is $T$-periodic.  The same
comparison argument shows that if $0<\epsilon_1 < \epsilon_2$, and for
$i=1,2$, $x_i(t)$ is the solution
of~\eqref{eq:constraint_generator:proof} with $\epsilon = \epsilon_i$,
then $x_1(t_0+T) < x_2(t_0+T)$, so that $\cS(\epsilon_1) <
\cS(\epsilon_2)$. In other words, $\cS(\epsilon)$ is a monotonically
increasing function, and so the value of $\epsilon$ above is
unique. This concludes the proof of part (i).  Now suppose that $K^+ /
K^- < (\cos x_0)^{-1}$, and let $\epsilon \in [\epsilon^-,\epsilon^+]$
be such that the solution $x(t)$
of~\eqref{eq:constraint_generator:proof} with initial condition
$x(t_0)=x_0$ is $T$-periodic, as in part (i). Let $x^+ =
\cos^{-1}\big(\frac{K^-}{K^+} c_{x_0}\big)$, and consider the subset
of the extended phase space $S^+= \{(x,t): x > x^+\}$. This subset is
positively invariant since $\tilde a>0$ and
\[
\begin{aligned}
\dot x\big|_{x=x^+} &= -\tilde a(t) \frac{K^-}{K^+} \cos x_0 +
\epsilon \tilde \mu(t) \\
&\hspace*{-5ex}\geq -\tilde a(t)
\frac{K^-}{K^+} \cos x_0 + (\cos x_0 / K^+) \tilde \mu(t)\\
&\hspace*{-5ex}\geq \frac{\tilde a(t)}{K^+} \cos x_0 \left( -K^- +
\tilde \mu(t)/\tilde a(t)\right)\\
&\hspace*{-5ex}\geq 0.
\end{aligned}
\]
Similarly, letting $x^- = \cos^{-1}\big(\frac{K^+}{K^-} c_{x_0}\big)$,
the subset $S^-= \{(x,t): x < x^-\}$ is positively invariant. Since
$\{ (x,t): x= x_0\}$ has empty intersection with $S^+ \cup S^-$, and
since $S^+ \cup S^-$ is positively invariant, it follows that the
$T$-periodic solution $x(t)$ with initial condition $x(t_0)=x_0$ must
be contained in the complement of $S^+ \cup S^-$, i.e., for all
$t\in\Re$, $x^- \leq x(t) \leq x^+$. For, if for some $\bar t \in
\Re$, $x(\bar t) \in S^+ \cup S^-$, then the fact that $x(t) \in S^+
\cup S^-$ for all $t\geq \bar t$ would contradict the periodicity of
$x(t)$.  
\end{proof}
\begin{example}\label{ex:ellipse}
{\rm Suppose $\C$ is an ellipse with major semiaxis $A$, minor
  semiaxis $B$, and $2 \pi$-periodic parametrization $\sigma(s) = (A
  \cos s, B \sin s)$, with $A=15$, $B=10$. The curvature is
  {$\kappa(s) = AB / ( A^2 \sin^2 s + B^2 \cos^2 s
    )^{3/2}$}. For the initial condition of the VHC
  generator~\eqref{eq:constraint_generator}, we pick $x(0) = \pi/8$.
  Following Proposition~\ref{prop:generator:periodic_solutions}, we
  need to choose a $2 \pi$-periodic function $\tilde \mu(t)>0$, set
  $\tilde \delta(t) = \epsilon \tilde \mu(t)$, and find the unique
  value of $\epsilon>0$ guaranteeing that the solution with initial
  condition $x(0) = \pi/8$ is $2\pi$-periodic.  There is much freedom
  in the choice of $\tilde \mu(t)$. For instance, picking $\tilde
  \mu(t) = 1$, we numerically find {$\epsilon \approx
    0.6482$}. The corresponding VHC is depicted in
  Figure~\ref{fig:ellipse:virtual_constraint}. The condition, in
  Proposition~\ref{prop:generator:periodic_solutions}(ii), that $K^+/
  K^- < (\cos x_0)^{-1}$ is conservative. Indeed, with our choice of
  $\tilde \mu$ we have $K^+ =3/2$, $K^-=2/3$, and thus the condition
  is violated. Yet, we see in
  Figure~\ref{fig:ellipse:virtual_constraint} that $ \Phi(s) \in (0,
  \pi/2)$ for all $s \in [\Re]_T$.
\begin{figure}
\psfrag{s}{$s$}
\psfrag{v}{$\varphi$}
\psfrag{P}{\small $\varphi=\Phi(s)$}
\centerline{\includegraphics[width=.495\textwidth]{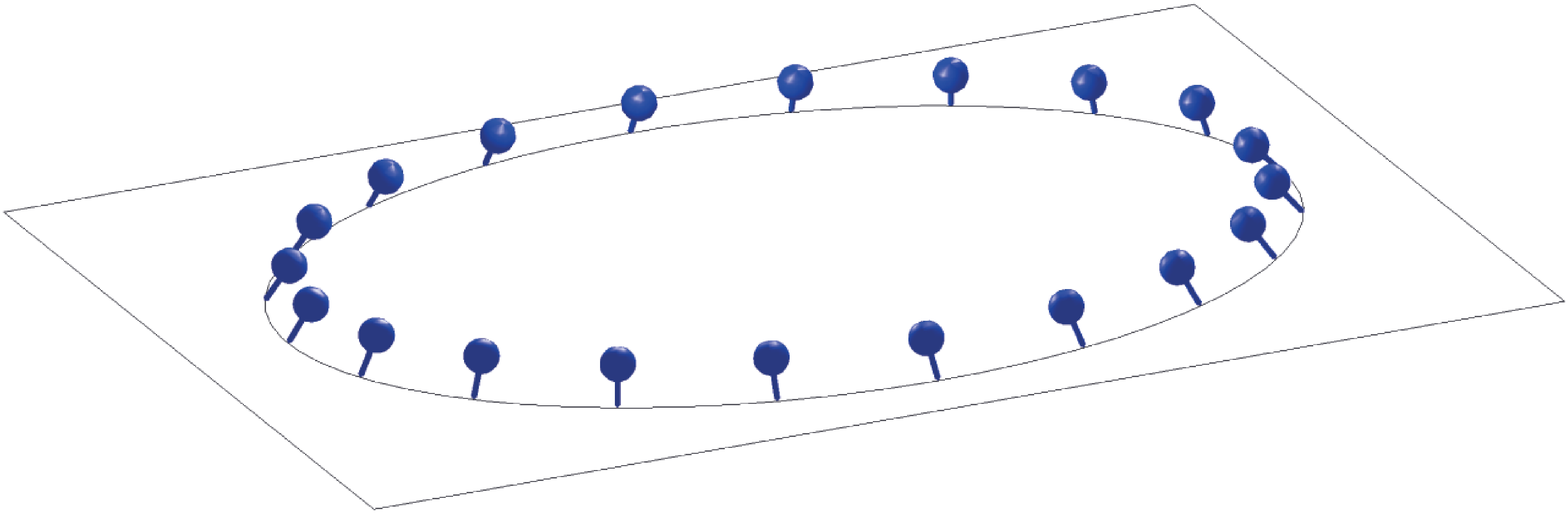}
  \includegraphics[width=.495\textwidth]{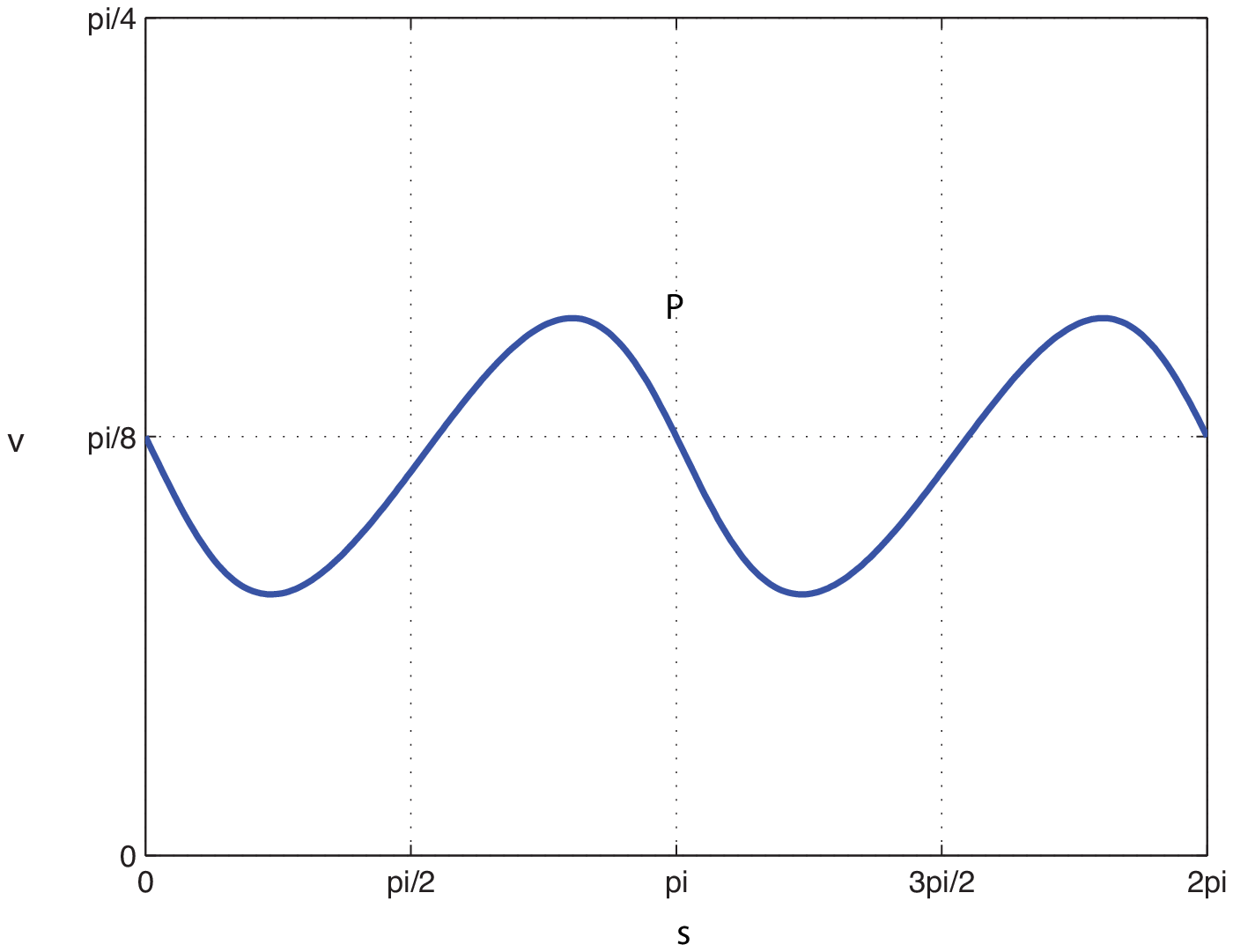}
}
\caption{VHC for the ellipse in Example~\ref{ex:ellipse}.}
\label{fig:ellipse:virtual_constraint}
\end{figure}
}
\end{example}

\section{Motion on the constraint manifold}\label{sec:constrained_motion}
Having chosen $\tilde \delta(t)= \epsilon \tilde \mu(t)$ as in
Proposition~\ref{prop:generator:periodic_solutions}(ii), and having
obtained a VHC $\varphi= \Phi(s)$, the next step is to analyse the
{\em reduced dynamics} of~\eqref{eq:sys} on the constraint manifold
$\Gamma = \{ (q,\dot q) \in \cX : \varphi = \Phi(s), \ \dot \varphi =
\Phi'(s) \dot s\}$. These are the zero dynamics of~\eqref{eq:sys} with
output function $H(\varphi,s)=\varphi - \Phi(s)$. To this end, we
impose that $[(d/dt) (\Phi'(s) \dot s)] |_\Gamma = \ddot
\varphi|_\Gamma$.  Expanding both sides of the equation above, using
the expression of $\ddot \varphi$ in~\eqref{eq:sys},
identity~\eqref{eq:Phiprime}, and the fact that $\delta \neq 0$, we
obtain the feedback making $\Gamma$ invariant
\begin{equation}
\label{eqn_inv_feedback}
\begin{aligned}
& u = \frac{h^{-1}g \sin \Phi}{\delta} - \frac{ \dot s^2
}{\delta} \Big[ \Phi'' + \frac 1 h ( (1-h \kappa \sin \Phi) \kappa
  \|\sigma'\| \\
& + b \kappa' +b\kappa {\sigma'}\trans\sigma'' / \|\sigma'\|^2)
  \|\sigma'\| \cos \Phi \Big].
\end{aligned}
\end{equation}
Substituting feedback~\eqref{eqn_inv_feedback} in the $s$ dynamics, we
get the reduced dynamics on $\Gamma$
\begin{equation}\label{eq:motion_on_Gamma}
\ddot s = \Psi_1(s) + \Psi_2(s) \dot s^2,
\end{equation}
with
\begin{equation}\label{eq:Psi}
\begin{aligned}
&\Psi_1 = \frac{h^{-1}g s_{\Phi}}{\delta} \\
& \Psi_2= - \frac{1}{\delta} \Big[ \Phi'' + \frac 1 h ( (1-h
    \kappa s_{\Phi}) \kappa\|\sigma'\| \\ & + b \kappa'
    + b \kappa \sigma'{}\trans \sigma'' / \|\sigma'\|^2 )
    c_{\Phi} \|\sigma'\| \Big].
\end{aligned}
\end{equation}
System~\eqref{eq:motion_on_Gamma} describes the motion of
system~\eqref{eq:sys} on the constraint manifold $\Gamma$ in the
following sense. For a given initial condition $(s(0),\dot s(0)) =
(s_0,\dot s_0) \in [\Re]_T \times \Re$, let $(s(t),\dot s(t))$ be the
corresponding solution of~\eqref{eq:motion_on_Gamma}, and let
$(\varphi(t),\dot \varphi(t), s_1(t),\dot s_1(t))$ be the solution
of~\eqref{eq:sys} with initial condition $(\varphi(0),\dot
\varphi(0),s(0),\dot s(0)) = ( \Phi(s_0), \Phi'(s_0) \dot s_0,s_0,\dot
s_0) \in \Gamma$. Then, for all $t \geq 0$ $(s_1(t),\dot s_1 (t)) =
(s(t),\dot s(t))$, and $(\varphi(t),\dot \varphi(t)) =$ $(\Phi(s(t)),$
$\Phi'(s(t)) \dot s(t))$.  In order for the VHC $\varphi=\Phi(s)$ to
be compatible with the maneuvering problem, we need to verify whether
or not the rear wheel of the bicycle traverses the entire curve $\C$
with bounded speed, i.e., we need to show that for any solution
$(s(t),\dot s(t))$ of~\eqref{eq:motion_on_Gamma}, there exist $\bar
t>0$ and $\epsilon>0$ such that $\dot s(t)>\epsilon > 0$ for all $ t
\geq \bar t$, and $\lim\sup_{t \to \infty} \dot s(t) < \infty$. The
next result explores general properties of systems of the
form~\eqref{eq:motion_on_Gamma}. %Recall that $\tilde h:=h\circ
%\pi$.
%
%
\begin{proposition}\label{prop:internal_dynamics_general_properties}
Consider a differential equation of the
form~\eqref{eq:motion_on_Gamma} with state space $\{(s,\dot s) \in
[\Re]_T \times \Re\}$. Assume that $\Psi_1, \Psi_2:[\Re]_T \to \Re$
are smooth functions such that $\Psi_1(s) > 0$ for all $s$ and
$\int_0^T \tilde \Psi_2(\tau) d\tau <0$. Then, there exists a smooth
function $\nu:[\Re]_T \to (0,\infty)$ such that the closed orbit $\cR
= \{(s,\dot s) \in [\Re]_T \times \Re: \dot s = \nu(s)\}$ is
exponentially stable for~\eqref{eq:motion_on_Gamma}, with domain of
attraction containing the set $\cD = \{(s,\dot s) \in [\Re]_T \times
\Re: \dot s \geq 0\}$.  {Thus,} for all initial conditions in
$\cD$, the solution $(s(t),\dot s(t))$ of~\eqref{eq:motion_on_Gamma}
converges to the unique asymptotically stable limit cycle $\cR$.
\end{proposition}
\begin{remark}\label{rem:doa}
{\rm It can be shown that the domain of attraction of the limit cycle
  $\cR$  is $\{(s,\dot s) : \dot s > -
  \nu(s)\}$.}
\end{remark}
\begin{remark}\label{rem:shiriaev}
{\rm Proposition~\ref{prop:internal_dynamics_general_properties} has
  general implications. Consider a Lagrangian control system of the
  form~\eqref{eq:lagrangian} with configuration vector $q=
  (q_1,\ldots, q_n)$ and degree of underactuation one. Consider a VHC
  of the form $q_i = \Phi_i(q_n)$, $i=1,\ldots,n-1$, and suppose $q_n
  \in [\Re]_T$. For this system, it is shown
  in~\cite{WesGriKod03,ShiPerWit05,MagCon13} that the reduced dynamics
  have the form~\eqref{eq:motion_on_Gamma}, with $s$ replaced by
  $q_n$. In this context,
  Proposition~\ref{prop:internal_dynamics_general_properties} gives
  sufficient conditions under which the reduced dynamics have an
  asymptotically stable closed orbit, implying that they are not
  Lagrangian. On the other hand, in Theorem~2
  of~\cite{Shiriaev2006900} (see also~\cite{ShiPerWit05}) it is shown
  that for any $(s_0,\dot s_0)$, system~\eqref{eq:motion_on_Gamma}
  possesses an ``integral of motion''\footnote{An integral of motion
    of~\eqref{eq:motion_on_Gamma}, or first integral, is a real-valued
    function of the state whose value is constant along all solutions
    of~\eqref{eq:motion_on_Gamma} (see, e.g.,~\cite{Har02}). The
    function $I(s,\dot s)$ is only constant along the solution
    of~\eqref{eq:motion_on_Gamma} with initial condition $(s(0),\dot
    s(0)) = (s_0,\dot s_0)$. It is not constant along other solutions,
    because the Lie derivative of the function $I(s,\dot s)$ along the
    vector field of~\eqref{eq:motion_on_Gamma} is not zero. In our
    previous work~\cite{MagCon13} we have shown that if $M$ and $V$
    are well-defined, then the correct integral of motion
    of~\eqref{eq:motion_on_Gamma} is $E(s,\dot s) = (1/2) M(s) \dot
    s^2 + V(s)$.} $I(s,\dot s) = \dot s^2 +2 V(s)/ M(s) - \dot
  s_0^2/M(s)$, where $ M(s) = \exp \left\{ -2 \int_{s_0}^s
  \Psi_2(\tau) d \tau \right\}$, $V(s) = -\int_{s_0}^s \Psi_1(\mu)
  M(\mu) d\mu$.  This integral of motion predicted by Theorem~2
  of~\cite{Shiriaev2006900} seemingly contradicts
  Proposition~\ref{prop:internal_dynamics_general_properties}. Indeed,
  if $M(s)$ and $V(s)$ are well-defined functions, letting $L(s,\dot
  s) = (1/2) M(s) \dot s^2 - V(s)$ one has that $\frac{d}{dt}
  \frac{\partial L}{\partial \dot s} - \frac{\partial L}{\partial s} =
  0$, and hence the reduced dynamics~\eqref{eq:motion_on_Gamma} are
  Lagrangian. This fact rules out the existence of isolated closed
  orbits for~\eqref{eq:motion_on_Gamma} predicted by
  Proposition~\ref{prop:internal_dynamics_general_properties}.  This
  contradiction is due to the fact that Theorem~2
  in~\cite{Shiriaev2006900} implicitly assumes that the functions $M(s)$
  and $V(s)$ are well-defined. This is always true if $s \in
  \Re$. However, when $s \in [\Re]_T$, the functions $M(s)$ and $V(s)$
  may be multi-valued, in which case the Lagrangian $L(s,\dot s)$ is
  undefined. Indeed, the assumptions of
  Proposition~\ref{prop:internal_dynamics_general_properties} imply
  that both $M(s)$ and $V(s)$ are multi-valued. We refer the reader to
  Section IV of~\cite{MagCon13} for a detailed discussion on this
  subject, and to~\cite{MohMagCon13} for necessary and sufficient
  conditions under which~\eqref{eq:motion_on_Gamma} is Lagrangian.}
\end{remark}
\begin{remark}\label{rem:Lagrange}
{\rm In Lagrangian mechanics, the Lagrange-d'Alembert principle
  implies that the enforcement of an ideal holonomic constraint (i.e.,
  a constraint for which the constraint forces do not produce work) on
  a Lagrangian system gives rise to a reduced Lagrangian
  system. Proposition~\ref{prop:internal_dynamics_general_properties}
  shows that this result is not true when enforcing VHCs on Lagrangian
  control systems. There is, therefore, a sharp difference between
  {\em ideal} and {\em virtual} holonomic constraints when it comes to
  the reduced motion they induce. This difference is due to the fact
  that VHCs are enforced through forces produced by feedback
  control. For underactuated systems such as the constrained Getz's
  bicycle model in~\eqref{eq:sys}, the control forces produce work.  }
\end{remark}

\begin{remark}{\rm
In the context of walking robots,~\cite{WesGriKod03} proved that under
certain conditions, the hybrid zero dynamics subject to a VHC exhibit
an exponentially stable hybrid limit cycle. The mechanism enabling
this exponential stability property is the jump map representing the
impact of the robot's foot with the
ground. Proposition~\ref{prop:internal_dynamics_general_properties}
shows that VHCs may induce stable limit cycles even when the control
system has no jumps.}
\end{remark}

%% \begin{remark}\label{rem:shiriaev}
%% {\rm Theorem~2 of~\cite{Shiriaev2006900} considers a system of the
%%   same form of~(\ref{eq:motion_on_Gamma}) (setting $f(\dot q)=\dot q$
%%   and $\alpha(q)=1$). This theorem proves that the system admits a
%%   first integral of motion and provides a closed form expression for
%%   it.  Apparently, this result is in conflict with
%%   Proposition~\ref{prop:internal_dynamics_general_properties}, since
%%   the existence of a first integral of motion for
%%   system~(\ref{eq:motion_on_Gamma}) contradicts the existence of
%%   attractive limit cycles. Note, however, that the integral of motion
%%   presented in~(\ref{eq:motion_on_Gamma}) assumes that the
%%   configuration variable is defined on $\mathbb{R}$ and is in general
%%   not valid for system~(\ref{eq:motion_on_Gamma}), whose configuration
%%   variable is defined on the set $[\Re]_T$ diffeomorphic to the unit
%%   circle
%%   $S^1$. Proposition~\ref{prop:internal_dynamics_general_properties}
%%   shows that, in general, system~(\ref{eq:motion_on_Gamma}) does not
%%   have a first integral. As shown in~\cite{MagCon12}, a sufficient
%%   condition for the existence of a first integral
%%   for~(\ref{eq:core_subsystem}) is that functions $\Psi_1$ and
%%   $\Psi_2$ are odd.}
%% \end{remark}
%
%
\begin{proof}
Consider the differential equation
\begin{equation}\label{eq:constrained_system_lifted}
\ddot x = \tilde \Psi_1 (x) + \tilde \Psi_2(x) \dot x^2.
\end{equation}
Let $\bar \pi: \Re \times \Re \to [\Re]_T \times \Re$ be defined as
$\bar \pi(x,\dot x) = (\pi(x),\dot x)$. Then, it is readily seen that
the vector fields in~\eqref{eq:motion_on_Gamma}
and~\eqref{eq:constrained_system_lifted} are $\bar
\pi$-related~\cite{Lee13}. It follows that if $(x(t),\dot x(t))$ is a
solution of~\eqref{eq:constrained_system_lifted}, then $(s(t),\dot
s(t)) = (\pi(x(t)),\dot x(t))$ is a solution
of~\eqref{eq:motion_on_Gamma}.  We will show that there exists a
smooth $T$-periodic function $\tilde \nu : \Re \to (0,+\infty)$ such
that the set $\tilde \cR = \{(x,\dot x) \in \Re^2: \dot x = \tilde
\nu(x)\}$ is an exponentially stable orbit
of~\eqref{eq:constrained_system_lifted}, with domain of attraction
containing $\tilde \cD = \{(x,\dot x)\in \Re^2: \dot x\geq 0\}$. Then,
setting $\nu = \tilde \nu \circ \pi^{-1}$,
by~\cite[Theorem~4.29]{Lee13} we will obtain a smooth function
$[\Re]_T \to (0,+\infty)$, and the set $\cR= \{(s,\dot s) \in [\Re]_T
\times \Re: \dot s = \nu(s)\}$ will be an exponentially stable
closed-orbit of~\eqref{eq:motion_on_Gamma} with domain of attraction
containing $\cD$, proving the proposition.  The set $\tilde \cD$ is
positively invariant for~\eqref{eq:constrained_system_lifted} because,
by assumption, $\ddot x|_{\dot x=0} = \tilde \Psi_1(x) >0$ for all
$x\in \Re$.  In the rest of the proof we will restrict initial
conditions to lie in $\tilde \cD$. Letting\footnote{This substitution
  is standard. See, for instance,~\cite[Section 2.9.3-2, item
    25]{Pol03}. It is also used in~\cite{Shiriaev2006900}.} $z = \dot
x^2$, we have $\dot z = 2\dot x (\tilde \Psi_1(x) + \tilde \Psi_2(x)
z)$. For all $(x,\dot x) \in \tilde \cD$, we have $\dot x>0$, so we
can use $x$ as a time variable:
\begin{equation}\label{eq:core_subsystem}
\frac { d z } {dx}  = 2\tilde \Psi_1(x) + 2\tilde \Psi_2(x) z.
\end{equation}
The above is a scalar linear $T$-periodic system. Letting $\phi(x) =
\exp(2\int_0^x \tilde \Psi_2(\tau) d \tau)$, the solution
of~\eqref{eq:core_subsystem} with initial condition $z(x_0)=z_0$ is
\[
z(x) = \phi(x) \phi(x_0)^{-1} z_0 + 2\int_{x_0}^x \phi(x) \phi^{-1}
(\tau) \tilde \Psi_1(\tau) d \tau.
\]
Note that, for any integer $k$, $\phi(x+k T) = \phi(T)^k \phi(x)$.
System~\eqref{eq:core_subsystem} has a $T$-periodic solution if and
only if there exists $\bar z_0$ such that the solution $z(x)$
of~\eqref{eq:core_subsystem} with initial condition $z(x_0) = \bar
z_0$ satisfies, $z(x_0+T) = \bar z_0$, or
\begin{equation}\label{eq:z0}
\bar z_0 = \phi(T) \bar z_0 + 2 \int_{x_0}^{x_0+T}\phi(x_0+T)
\phi^{-1}(\tau) \tilde \Psi_1(\tau) d \tau.
\end{equation}
By assumption, $0< \phi(T) <1$, so there is a unique $\bar z_0>0$
solving~\eqref{eq:z0}, implying that there is a unique $T$-periodic
solution $\bar z:\Re \to \Re$ of~\eqref{eq:core_subsystem}. Since for
all $x\in \Re$, $z=0 \implies (d z/ dx) = 2\tilde \Psi_1(x)>0$, the
set $\{(z,x): z>0\}$ is positively invariant
for~\eqref{eq:core_subsystem}, and therefore the $T$-periodic function
$\bar z$ satisfies $\image(\bar z) \subset (0,\infty)$.  Let $z:\Re
\to \Re$ be any other arbitrary solution of~\eqref{eq:core_subsystem}
(therefore, not necessarily $T$-periodic). Let $k$ be a nonnegative
integer and denote $z_k:=z(x_0+k T)$, $\phi_k:=\phi(x_0+kT)$,
$x_k:=x_0+kT$. Then,
\[
\begin{aligned}
&z_{k+1} = \phi_{k+1} \phi_k^{-1} z_k + 2\int_{x_k}^{x_{k+1}}
  \phi_{k+1} \phi^{-1}(\tau) \tilde \Psi_1(\tau) d \tau \\
& = \phi(T)  z_k +2\int_{x_0}^{x_1} \phi_1 \phi(T)^k
\phi^{-1} (\tau + kT) \tilde \Psi_1(\tau + kT) d \tau \\
& = \phi(T)  z_k +2\int_{x_0}^{x_1} \phi_1
\phi^{-1}(\tau) \tilde \Psi_1(\tau) d \tau \; \text{($\tilde
  \Psi_1$ is $T$-periodic)}\\
& = \phi(T)  ( z_k - \bar z_0) + \bar z_0. \quad
\text{(by~\eqref{eq:z0})} 
\end{aligned}
\]
In light of the above, letting $\tilde z_k = z_k -\bar z_0$, we
have $\tilde z_{k+1} = \phi(T) \tilde z_k$. Since $0<\phi(T)<1$, the
origin of this discrete-time system is globally exponentially stable,
proving that the $T$-periodic solution $\bar z:\Re\to(0,\infty)$ is
globally exponentially stable for~\eqref{eq:core_subsystem}. Define
$\tilde \nu:\Re \to (0,\infty)$ as $\nu:= \sqrt{\bar z}$, and return
to system~\eqref{eq:motion_on_Gamma}. For an arbitrary initial
condition $(x(0),\dot x(0)) \in \tilde \cD$, the solution $(x(t),\dot
x(t))$ satisfies $\dot x(t) > 0$ for all $t >0$, and $\dot x(t)
=\sqrt{ z(x(t))}$, where $z:\Re\to\Re$ is the solution
of~\eqref{eq:core_subsystem} with initial condition $z(x(0)) =(\dot
x(0))^2$. Since the solution $\bar z:\Re \to(0,\infty)$ is globally
exponentially stable for~\eqref{eq:core_subsystem}, the set $\tilde
\cR$ is exponentially stable for~\eqref{eq:constrained_system_lifted}
with domain of attraction containing $\tilde \cD$. Going back to
$(s,\dot s)$ coordinates, this implies that $\cR = \{(s,\dot s) \in
[\Re]_T \times \Re: \dot s = \nu(s)\}$ is exponentially stable
for~\eqref{eq:motion_on_Gamma}, and its domain of attraction contains
$\cD$. Note that $\cR$ is a simple closed curve in $[\Re]_T \times
\Re$.  We are left to show that $\cR$ is a closed orbit
of~\eqref{eq:motion_on_Gamma}. The set $\cR$ is closed, invariant, and
no proper subset of it has these properties. For, if there existed a
closed and invariant proper subset $\cR' \subset \cR$, then any
solution $(s(t),\dot s(t))$ originating in $\cR'$ would not traverse
the entire curve $\cR$, contradicting the fact that, for all $(s,\dot
s)\in\cR$, $\dot s >0$. $\cR$ is, therefore, a minimal set
for~\eqref{eq:motion_on_Gamma}. By~\cite[Theorem 12.1]{Har02}, $\cR$
is a closed orbit.
\end{proof}
We now show that if $\C$ is sufficiently long as compared to $b$ and
$h$, the bicycle satisfies the hypotheses of
Proposition~\ref{prop:internal_dynamics_general_properties}, and so
the reduced motion~\eqref{eq:motion_on_Gamma} satisfies the
requirements of the maneuvering problem.
\begin{proposition}\label{prop:internal_dynamics_bicycle}
Assume that the  curvature of $\C$ satisfies the inequality
\begin{equation}\label{eq:curvature_bound}
\frac 1 L \int_0^T \tilde \kappa(\tau) \|\tilde \sigma'(\tau)\| d\tau
< \frac{h}{b^2+h^2},
\end{equation}
where $L=\int_0^T \|\tilde \sigma'(\tau)\| d \tau$ is the length of
$\C$. Then, any VHC $\varphi = \Phi(s)$ with $\image(\Phi) \subset
(0,\pi/2)$ yields functions $\Psi_1,\Psi_2:[\Re]_T \to \Re$
in~\eqref{eq:Psi} that satisfy the hypotheses of
Proposition~\ref{prop:internal_dynamics_general_properties}. Hence, in
particular, the reduced dynamics~\eqref{eq:motion_on_Gamma} have an
exponentially stable closed orbit.
\end{proposition}
\begin{remark}
{\rm The integral $(1/L)\int_0^T \tilde \kappa(\tau) \|\tilde
  \sigma'(\tau)\| d\tau$ is equal to ({\em turning number} of $\C)
  \times 2\pi/{L}$. {The turning number of $\C$ is the
    number of counterclockwise revolutions that the tangent vector to
    $\C$ makes as its base point is moved once around $\C$ in a way
    consistent with the orientation of $\C$. The turning number of a
    Jordan curve is $\pm 1$, and for curves satisfying
    Assumption~\ref{assump_convex} it is always $1$. Thus,
    condition~\eqref{eq:curvature_bound} can be written as $L > 2 \pi
    (b^2+h^2)/h$. It requires the curve to be sufficiently long as
    compared to the bicycle parameters $b$ and $h$.}}
\end{remark}
\begin{remark}\label{rem:not_EL}
{\rm Proposition~\ref{prop:internal_dynamics_bicycle} implies that
  when a VHC $\varphi = \Phi(s)$ with $\image(\Phi) \subset (0,\pi/2)$
  is enforced on the Lagrangian control system~\eqref{eq:sys}, the
  resulting reduced dynamics are not Lagrangian.}
\end{remark}

\begin{proof}
The VHC $\varphi = \Phi(s)$ arising from
Proposition~\ref{prop:generator:periodic_solutions}, part (ii),
satisfies $\Phi(s) \in (0, \pi/2)$ for all $s\in [\Re]_T$, so that
$\sin \Phi(s) >0$. Recall that $\Phi$ satisfies~\eqref{eq:Phiprime}
with $\delta(s) =\epsilon \mu(s)>0$ for all $s \in [\Re]_T$. This
implies that $\Psi_1(s)>0$. Moreover, $\Phi'' = \delta' -r \kappa'
\|\sigma'\| c_\Phi - (r \kappa\|\sigma'\|)^2 s_\Phi c_\Phi +
r \kappa \|\sigma'\|\delta s_\Phi -r \kappa {\sigma'}\trans
\sigma'' / \|\sigma'\|$.  Substituting this expression
in~\eqref{eq:Psi}, we get
\[
\begin{aligned} 
\Psi_2 &= -\frac{\delta'}{\delta} - \frac {\kappa\|\sigma'\|} {\delta
  h} \Big[ b \delta s_\Phi - \|\sigma'\|c_\Phi \big( \kappa s_\Phi
  \frac{b^2+h^2} h -1 \big) \Big] \\
&\leq -\frac{\delta'}{\delta} + \frac{\kappa\|\sigma'\|^2} {\delta h}
c_\Phi \big( \kappa s_\Phi \frac{b^2+h^2} h -1 \big) .
\end{aligned}
\]
Since $\int_0^T \tilde \delta'(\tau)/\tilde \delta(\tau) = \ln \tilde
\delta(T) - \ln \tilde \delta(0) =0$, using the fact that
$\image(\tilde \Phi) \subset (\Phi^-,\Phi^+) \subset (0,\pi/2)$, we have
\[
\begin{aligned}
&\int_0^T \tilde \Psi_2(\tau) d\tau \leq \int_0^T \frac{\tilde
    \kappa\|\tilde \sigma'\|^2} {\tilde \delta h} c_{\tilde \Phi}
  \left( \tilde \kappa s_{\tilde \Phi} \frac{b^2+h^2} h -1 \right)
  d\tau \\
&\leq \max_t \left(\frac{\tilde \kappa\|\tilde \sigma'\|} {\tilde
    \delta h}\right) c_{\Phi^-} \int_0^T \left(\tilde \kappa
  \frac{b^2+h^2} h -1 \right) \|\tilde \sigma'\| d\tau \\
&\leq \max_t \left(\frac{\tilde \kappa\|\tilde \sigma'\|^2} {\tilde
    \delta h}\right) c_{\Phi^-} \left( \frac{b^2+h^2} h \int_0^T
  \tilde \kappa \|\tilde \sigma'\| d\tau -L \right)\\
&<0.
\end{aligned} \]

\end{proof}
\begin{example}
{\rm We return to the ellipse of Example~\ref{ex:ellipse} and the VHC
  displayed in Figure~\ref{fig:ellipse:virtual_constraint}. For this
  example,
%{Questi numeri non sono corretti. I conti vanno rifatti.}
$(1/2\pi)\int_0^{2\pi} \tilde \kappa(\tau) d\tau \approx 0.0792$, and
  $h/(b^2+h^2)=0.6711$ and thus~\eqref{eq:curvature_bound} is
  satisfied. Indeed, one can numerically check that $\int_0^{2\pi} \tilde
  \Psi_2(\tau) d\tau \approx -105.1<0$, and
  Proposition~\ref{prop:internal_dynamics_general_properties} applies.
  The phase portrait of the dynamics on the constraint manifold is
  displayed in Figure~\ref{fig:ellipse:phase_portrait}. The figure
  illustrates the set $\cR$, corresponding to the steady-state
  velocity profile of the bicycle on $\Gamma$. The domain of
  attraction of $\cR$, shaded in the figure, is the set $\{(s,\dot s):
  \dot s > -\nu(s)\}$, as pointed out in Remark~\ref{rem:doa}.
\begin{figure}
\psfrag{s}{$s$} \psfrag{t}{$\dot s$}
\psfrag{R}{$\cR$}
%\centerline{\includegraphics[width=.7.5\textwidth]{FIGURES/ellipse_phase_portrait}}
\centerline{\includegraphics[width=.5\textwidth]{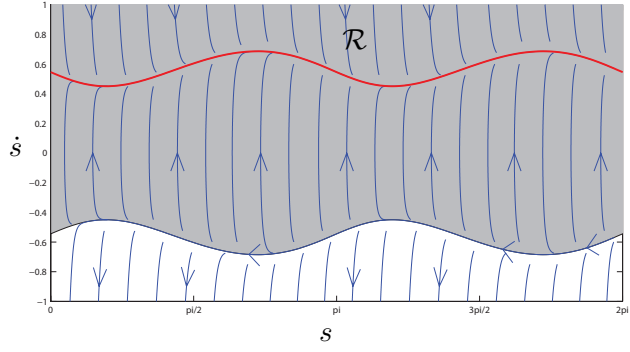}}
\caption{Phase portrait of the dynamics on $\Gamma$ and set $\cR$ for
  the ellipse in Example~\ref{ex:ellipse}. The shaded region is the
  domain of attraction of $\cR$. Since $s \in [\Re]_{2\pi}$, the $s$
  axis wraps around, and points on the lines $s = 0$ and $s = 2\pi$
  are identified in the figure, from which it follows that $\cR$ is a
  closed orbit.}
\label{fig:ellipse:phase_portrait}
\end{figure}
}

\end{example}

\begin{example}
{\rm Let now $\C$ be a circle of radius $R$ with parametrization
  $\sigma(s) =(R \cos s,R \sin s)$. We have $T=2\pi$, $\|\sigma'(s)\|
  \equiv R$ and $\kappa(s) \equiv 1/R$, and so $a(s) \equiv
  bh^{-1}$. For any $x_0 \in (0,\pi/2)$, picking $\tilde \delta= bh^{-1}
  \cos x_0$, as in Remark~\ref{rem:mu}, we obtain the constant VHC
  $\varphi=x_0$. Equation~\eqref{eq:core_subsystem} becomes
\[
\frac{dz}{dx} =\frac{2g}{b} \tan x_0 - \frac{2}{bR} ( 1-
(h/R)\sin x_0) z.
\]
The above is a linear time-invariant system with constant input which
is stable if $R > h \sin x_0$.  The periodic solution $\bar z(t)$ in
this case is simply the equilibrium of the system above, $\bar z= g
R^2 \tan x_0 / (R - h \sin x_0) $, and thus the asymptotic velocity of
the bicycle on $\Gamma$ is constant, and reads as $\nu = R \sqrt{g
  \tan x_0 / (R - h \sin \Phi_0)}$.  It can be verified that $\nu$ is
an increasing function of $x_0$. The conclusion is that the bicycle
can travel around the circle with any constant roll angle in the
interval $(0,\pi/2)$.  In steady-state, its speed is constant. The
larger is the roll angle $x_0$, the higher is the asymptotic speed of
the bicycle.  }
\end{example}

\section{Solution of the maneuvering problem}\label{sec:solution}
\begin{theorem}\label{thm:solution}
Suppose that the curvature of $\C$ satisfies
inequality~\eqref{eq:curvature_bound}. If $\varphi=\Phi(s)$ is a VHC
such that $\Phi(s) \in
(0,\pi/2)$ for all $ s \in [\Re]_T$, then the feedback
\begin{equation}\label{eq:stabilizing_feedback}
\begin{aligned}
u& =\frac{1}{\Delta(q)}\Big( \frac 1 h g \sa - \Big( \Phi'' + \frac 1
h\Big( (1-h \kappa \sa) \kappa \|\sigma'\|\\
&+ b \kappa'+\frac{b \kappa {\sigma'}\trans \sigma''}{\|\sigma'\|^2}\Big) \ca
\|\sigma'\|\Big) \dot s^2 + K_1 e + K_2 \dot e\Big),
\end{aligned}
\end{equation}
where $\Delta(q) = \Phi'(s) + a(s) \ca$,
$e= \varphi - \Phi(s)$, $\dot e =\dot \varphi - \Phi'(s)\dot s$, and
$K_1$, $K_2$ are positive design parameters, solves the maneuvering
problem and has the following properties:
\begin{enumerate}[(i)]
\item The constraint manifold $\Gamma$ is invariant and locally
  exponentially stable for the closed-loop
  system~\eqref{eq:sys},~\eqref{eq:stabilizing_feedback}.

\item There exists a $C^1$ function $\nu: [\Re]_T
  \to (0,\infty)$  such that the closed orbit $\bar \cR =
  \{(q,\dot q) \in \Gamma: \dot s=\nu(s)\}$ is asymptotically stable
  for the closed-loop system, and its domain of attraction is a
  neighbourhood of the set $\{(q,\dot q) \in \Gamma : \dot s>0\}$.

\item For initial conditions in the domain of attraction of $\bar
  \cR$, the rear wheel of the bicycle traverses the entire curve $\C$,
  and its speed is periodic in steady-state.

\end{enumerate}
\end{theorem}
\begin{remark}
{\rm Control law~(\ref{eq:stabilizing_feedback}) is {an input-output
    linearizing feedback for system~\eqref{eq:sys} with output
    $H(\varphi,s)=\varphi-\Phi(s)$. Letting $e= \varphi - \Phi(s)$,
    the closed-loop system satisfies $\ddot{e}=-K_1 e - K_2 \dot e$,
    and it exponentially stabilizes $\Gamma = \{(q,\dot q): e=\dot
    e=0\}$. Substituting $e=\dot e=0$
    in~(\ref{eq:stabilizing_feedback}) we recover the feedback
    in~(\ref{eqn_inv_feedback}) rendering $\Gamma$ invariant.}}
\end{remark}
\begin{proof}
The map $(\varphi,\dot \varphi,s,\dot s) \mapsto (e,\dot e,s,\dot s)$
is a diffeomorphism, and the image of $\Gamma$ under this map is the
set $\tilde \Gamma :=\{(e,\dot e,s,\dot s): e=\dot e=0\}$. The
feedback~\eqref{eq:stabilizing_feedback} is smooth in a neighbourhood
of $\Gamma$, and it yields $\ddot e + K_1 e + K_2 =0$, with
$K_1,K_2>0$. Therefore, the set $\tilde \Gamma$ is exponentially
stable, implying that $\Gamma$ is exponentially stable as well, and
proving part (i).  As for part (ii), by
Propositions~\ref{prop:internal_dynamics_general_properties}
and~\ref{prop:internal_dynamics_bicycle}, there exists a smooth
function $\nu:[\Re]_T \to (0,\infty)$ such that the closed orbit $\cR
= \{(s,\dot s): \dot s = \nu(s)\}$ is exponentially stable for
system~\eqref{eq:motion_on_Gamma}. Therefore, $\bar \cR$ is
asymptotically stable relative to $\Gamma$ (i.e., asymptotically
stable when the initial conditions are restricted to lie on
$\Gamma$). In order to prove that $\bar \cR$ is asymptotically stable
for initial conditions in a neighbourhood of $\Gamma$, note that $\bar
\cR$ is a closed curve, and hence a compact set.  Owing to the
reduction principle for asymptotic stability of compact sets
(see~\cite{SeiFlo95} and \cite[Theorem 10]{ElHMag13}), the asymptotic
stability of $\bar \cR$ relative to $\Gamma$, together with the
asymptotic stability of $\Gamma$, imply that $\bar \cR$ is
asymptotically stable for~\eqref{eq:sys}. By
Propositions~\ref{prop:internal_dynamics_general_properties}
and~\ref{prop:internal_dynamics_bicycle}, its domain of attraction
contains the set $\{(q,\dot q) \in \Gamma : \dot s>0\}$. Since the
domain of attraction of a closed set is an open set, the domain of
attraction of $\Gamma$ is a neighbourhood of $\{(q,\dot q) \in \Gamma
: \dot s>0\}$. Finally, concerning part (iii), for all $(q,\dot q)
\in\bar \cR$ we have $\dot s = \nu(s)>0$. Hence, for all initial
conditions in the domain of attraction of $\bar \cR$ there exist $\bar
t>0$ and $\epsilon >0$ such that $\dot s(t)>\epsilon>0$ for all $ t
\geq \bar t$, and hence the bicycle traverses the entire curve
$\C$. Since $\bar \cR$ is diffeomorphic to $S^1$, since it is
asymptotically stable, and solutions originating on it are periodic,
$\bar \cR$ is a stable limit cycle of the closed-loop
system. Therefore, solutions in the domain of attraction of $\bar \cR$
converge asymptotically to a periodic orbit.  
\end{proof}

\section{Numerical implementation}\label{sec:numerics}

To implement the solution presented in Theorem~\ref{thm:solution} one
needs an analytical expression of the function $\Phi: [\Re]_T \to
\Re$. This function is found through numerical integration of the
virtual constraint generator~\eqref{eq:constraint_generator} using a
numerical procedure. We will now outline this procedure and make
informal deductions about the impact of the approximation. We begin
our design by picking a function $\tilde \mu:\Re \to (0,\infty)$
according to Proposition~\ref{prop:generator:periodic_solutions}, part
(ii), and an initial condition $x_0$.  Using a one-dimensional search,
we find the unique value of $\epsilon$ such that the solution
of~\eqref{eq:constraint_generator} is $T$-periodic with a desired
accuracy. The function $\Phi:[\Re]_T \to \Re$ is determined as a
spline interpolation from the samples of the solution
of~\eqref{eq:constraint_generator} over a period. With a spline
interpolation at hand, one can compute $\Phi', \Phi''$ analytically
and implement~\eqref{eq:stabilizing_feedback}. The numerical
approximation process used to find $\Phi$ introduces a bounded error
in the feedback controller~\eqref{eq:stabilizing_feedback} which can
be made arbitrarily small. The effect of this error is to perturb the
constraint manifold $\Gamma$ without affecting its exponential
stability. The dynamics on the perturbed manifold are still governed
by~\eqref{eq:motion_on_Gamma}, where the functions $\Psi_1,\Psi_2$ are
affected by small perturbations. Such perturbations have no effect on
the hypotheses of
Proposition~\ref{prop:internal_dynamics_general_properties}, because
they involve strict inequalities, $\Psi_1 >0$, $\int_0^T \tilde \Psi_2
d \tau <0$. Thus, the conclusion of Theorem~\ref{thm:solution} still
holds. The approximation of $\Phi$ perturbs the constraint manifold
$\Gamma$ and the asymptotically stable limit cycle $\bar \cR$.

\begin{example}
{\rm We return to example of the ellipse, with the VHC depicted in
  Figure~\ref{fig:ellipse:virtual_constraint}. The simulation results
  for the closed-loop system with
  controller~\eqref{eq:stabilizing_feedback} and $K_1=5$, $K_2=2$ are
  shown in
  Figures~\ref{fig:ellipse:closed_loop2},~\ref{fig:ellipse:closed_loop1}
  for the initial condition $(\varphi(0),\dot \varphi(0),s(0),\dot
  s(0)) = (0.1,0,0,0.2)$. Figure~\ref{fig:ellipse:closed_loop2}
  illustrates the exponential convergence of $\varphi(t)- \Phi(s(t))$
  to zero. Figure~\ref{fig:ellipse:closed_loop1} displays the
  projection of the phase curve on the $(s,\dot s)$ plane and its
  convergence to the submanifold $\cR$.

\begin{figure}
\psfrag{t}{$t$}
\psfrag{P}{$\Phi(s(t))$}
\psfrag{Q}{$\varphi(t)$}
%\centerline{\includegraphics[width=.7.5\textwidth]{FIGURES/ellipse_closed_loop2}}
\centerline{\includegraphics[width=.5\textwidth]{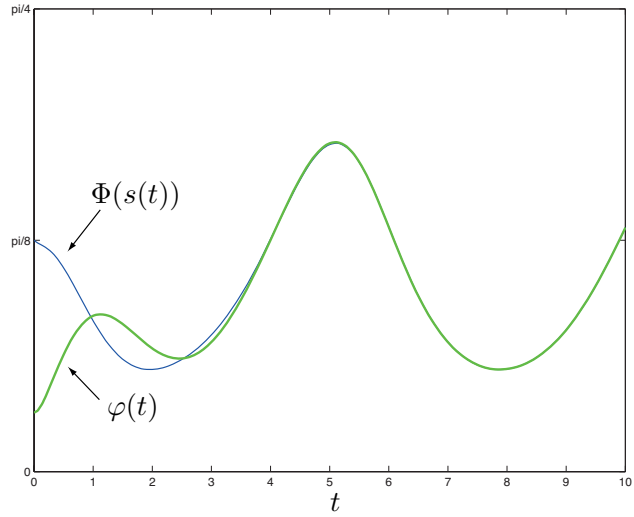}}
\caption{Simulation of the closed-loop system for the ellipse
  example. The solution converges to the constraint manifold
  $\Gamma$.}
\label{fig:ellipse:closed_loop2}
\end{figure}
\begin{figure}
\psfrag{s}{$s$} \psfrag{t}{$\dot s$}
\psfrag{R}{$\cR$}
%\centerline{\includegraphics[width=.7.5\textwidth]{FIGURES/ellipse_closed_loop1}}
\centerline{\includegraphics[width=.5\textwidth]{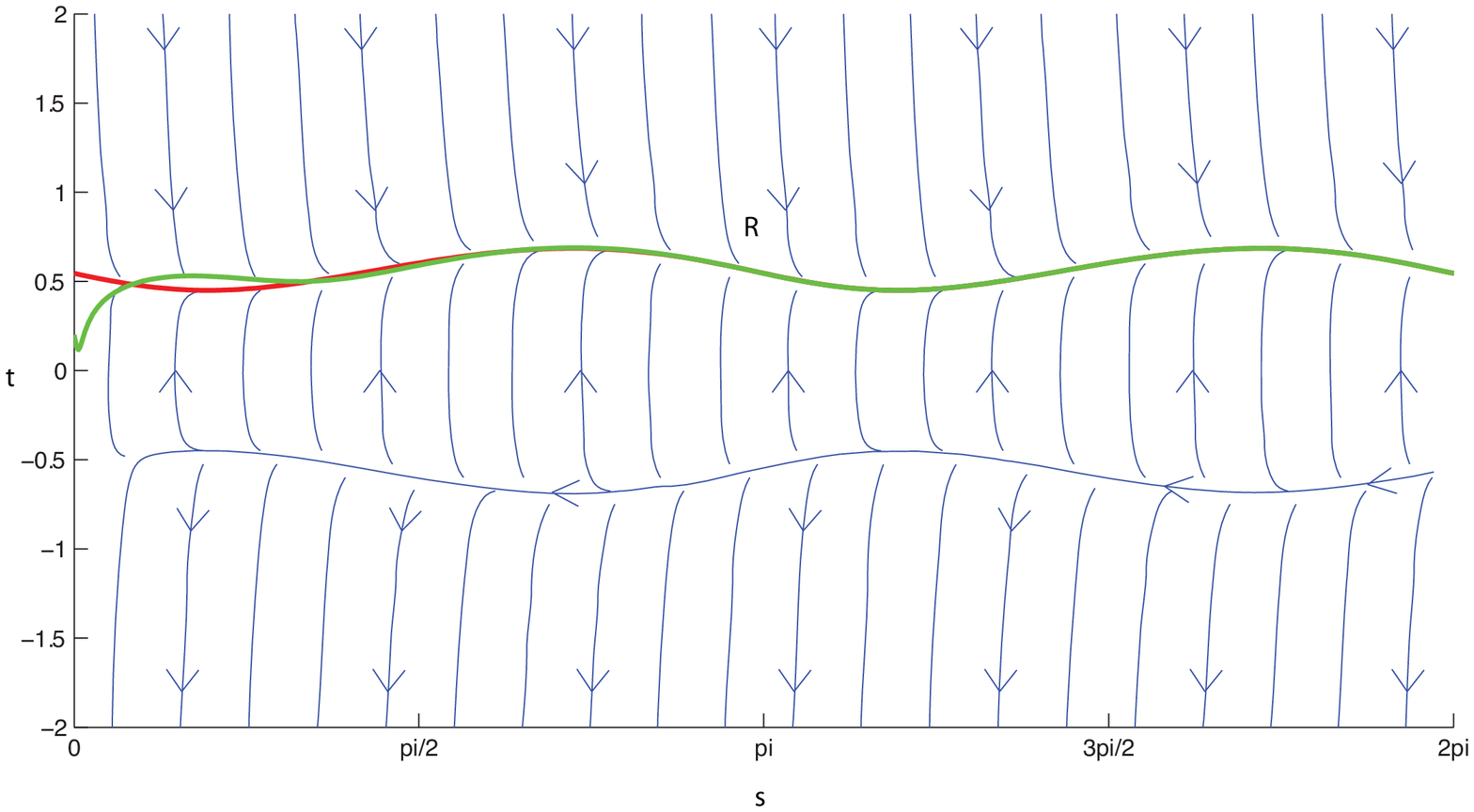}}
\caption{Simulation of the closed-loop system for the ellipse
  example. The solution converges to the submanifold $\cR \subset \Gamma$.}
\label{fig:ellipse:closed_loop1}
\end{figure}
}
\end{example}

\begin{example}
\label{example_non_convex}
{\rm In this second example, we consider the $2 \pi$-periodic curve
  shown in Figure~\ref{fig:non_convex}, parametrized as $(5+1.5 \cos
  (2 t)) (\cos t, \sin t)$, and then reparameterized with respect to the
  arc-length $s$. The curve has length $L=39.129$.  Since
  $\sigma:[\Re]_L \to \Re^2$ is not convex,
  Assumption~\ref{assump_convex} is not satisfied and the results
  presented in the paper cannot be applied directly. Nevertheless, the
  maneuvering problem can still be solved applying the VHC method. For
  the virtual constraint generator~(\ref{eq:constraint_generator}), we
  pick $\tilde \delta(t)\equiv \epsilon$, and choose $x(0)=0.35$.
  Numerically, we find that setting $\epsilon=0.1194$ the solution
  of~\eqref{eq:constraint_generator} is $L$-periodic, thus giving rise
  to a valid VHC.
\begin{figure}[H]
%\psfrag{s}{$s$} \psfrag{t}{$\dot s$}
%\psfrag{R}{$\cR$}
\centerline{\includegraphics[width=.5\textwidth]{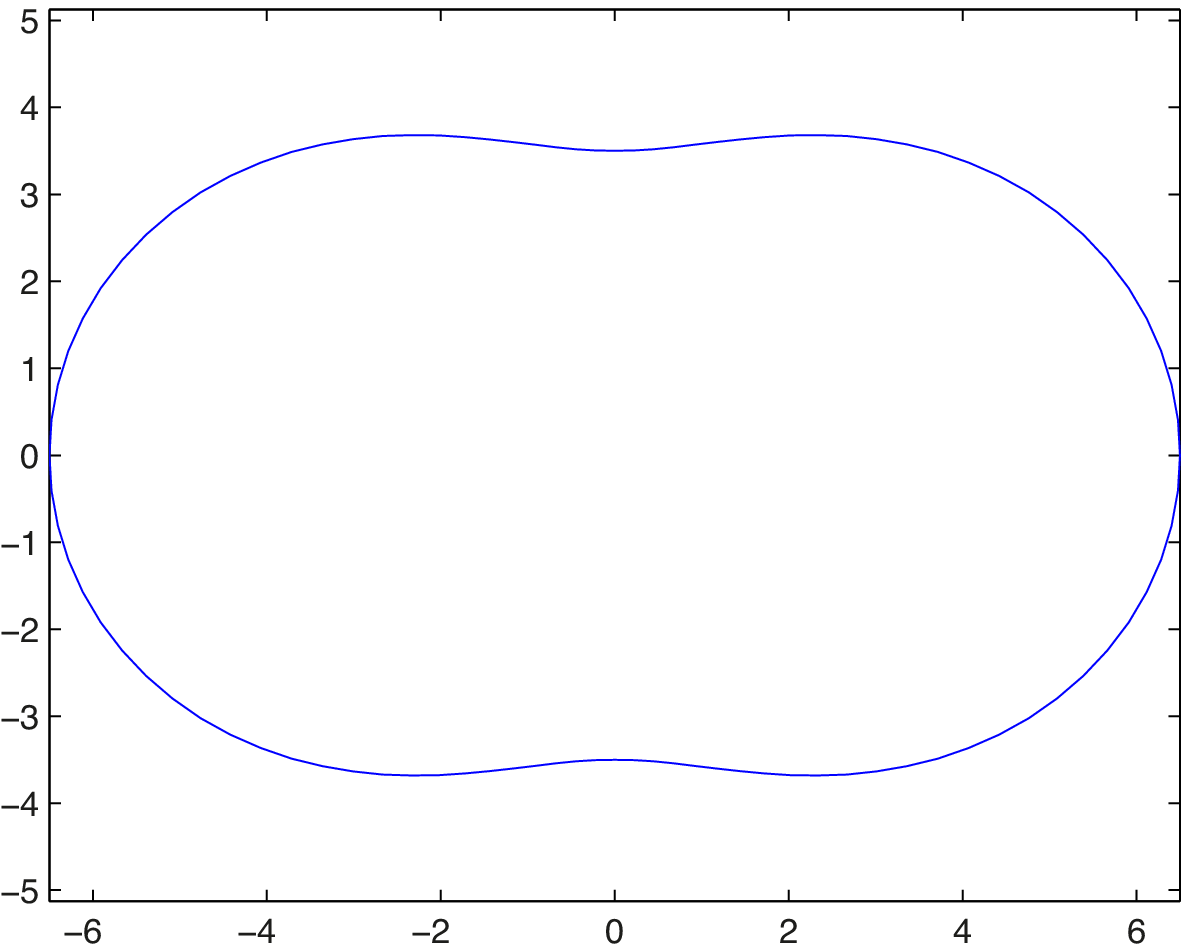}}
\caption{The non-convex curve considered in example~\ref{example_non_convex}.}
\label{fig:non_convex}
\end{figure}
Again, the simulation results for the closed-loop system with
controller~\eqref{eq:stabilizing_feedback} and $K_1=5$, $K_2=2$ are
shown in
Figures~\ref{fig:non_convex:closed_loop2},~\ref{fig:non_convex:closed_loop1}
for the initial condition $(\varphi(0),\dot \varphi(0),s(0),\dot s(0))
= (0,0,0,2)$.

%Figure~\ref{fig:non_convex:closed_loop2} illustrates the
%exponential convergence of $\varphi(t)$ to the constraint
%$\Phi(s(t))$. Figure~\ref{fig:non_convex:closed_loop1} displays the
%projection of the phase curve on the $(s,\dot s)$ plane and its
% convergence to the submanifold $\cR$. 
%
\begin{figure}[H]
\psfrag{t}{$t$}
\psfrag{P}{$\Phi(s(t))$}
\psfrag{Q}{$\varphi(t)$}
\centerline{\includegraphics[width=.5\textwidth]{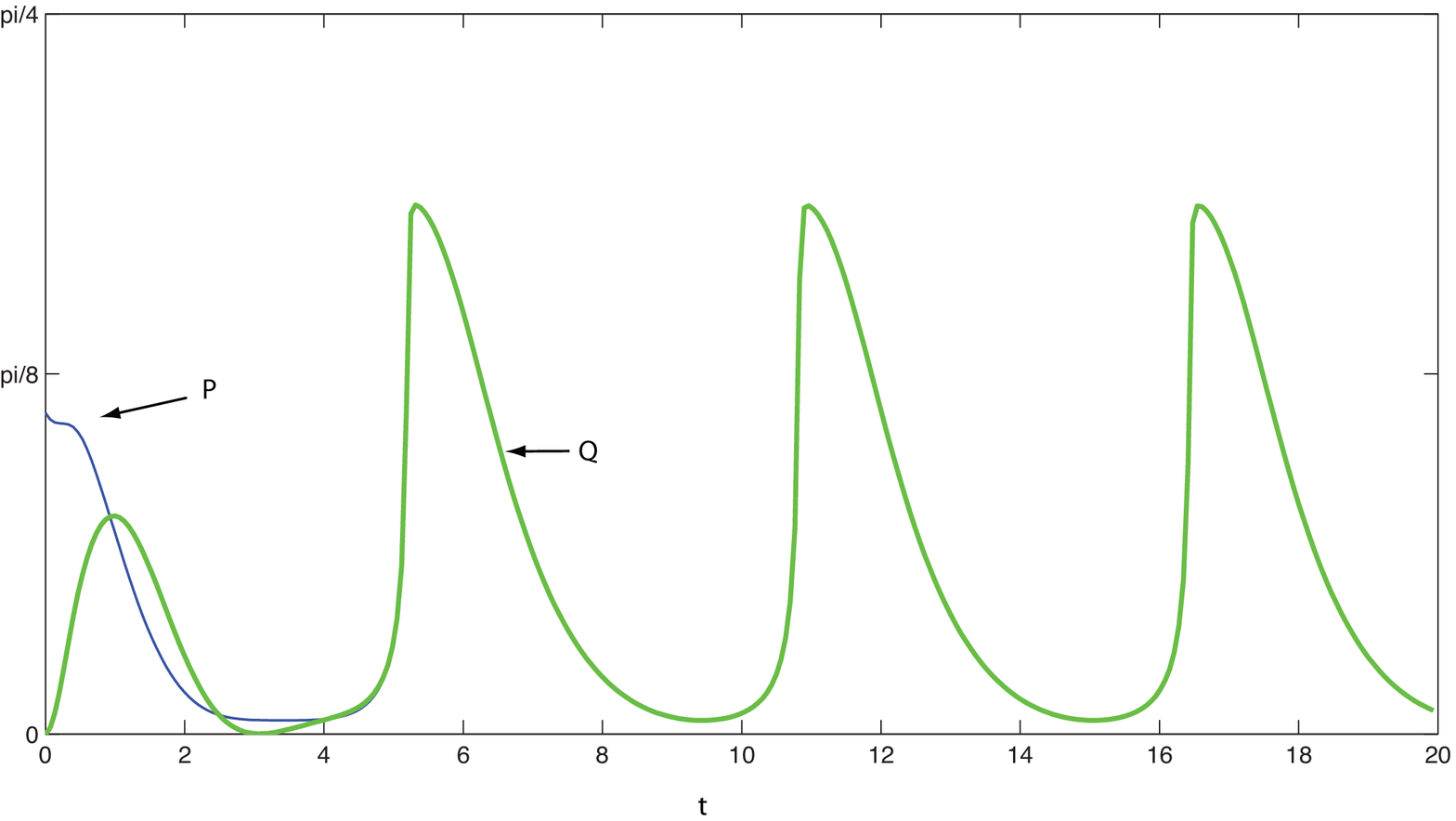}}
\caption{Simulation of the closed-loop system for the curve considered in
  example~\ref{example_non_convex}. The solution converges to the constraint manifold
  $\Gamma$.}
\label{fig:non_convex:closed_loop2}
\end{figure}
\begin{figure}[H]
\psfrag{s}{$s$} \psfrag{t}{$\dot s$}
\psfrag{R}{$\cR$}
\centerline{\includegraphics[width=.5\textwidth]{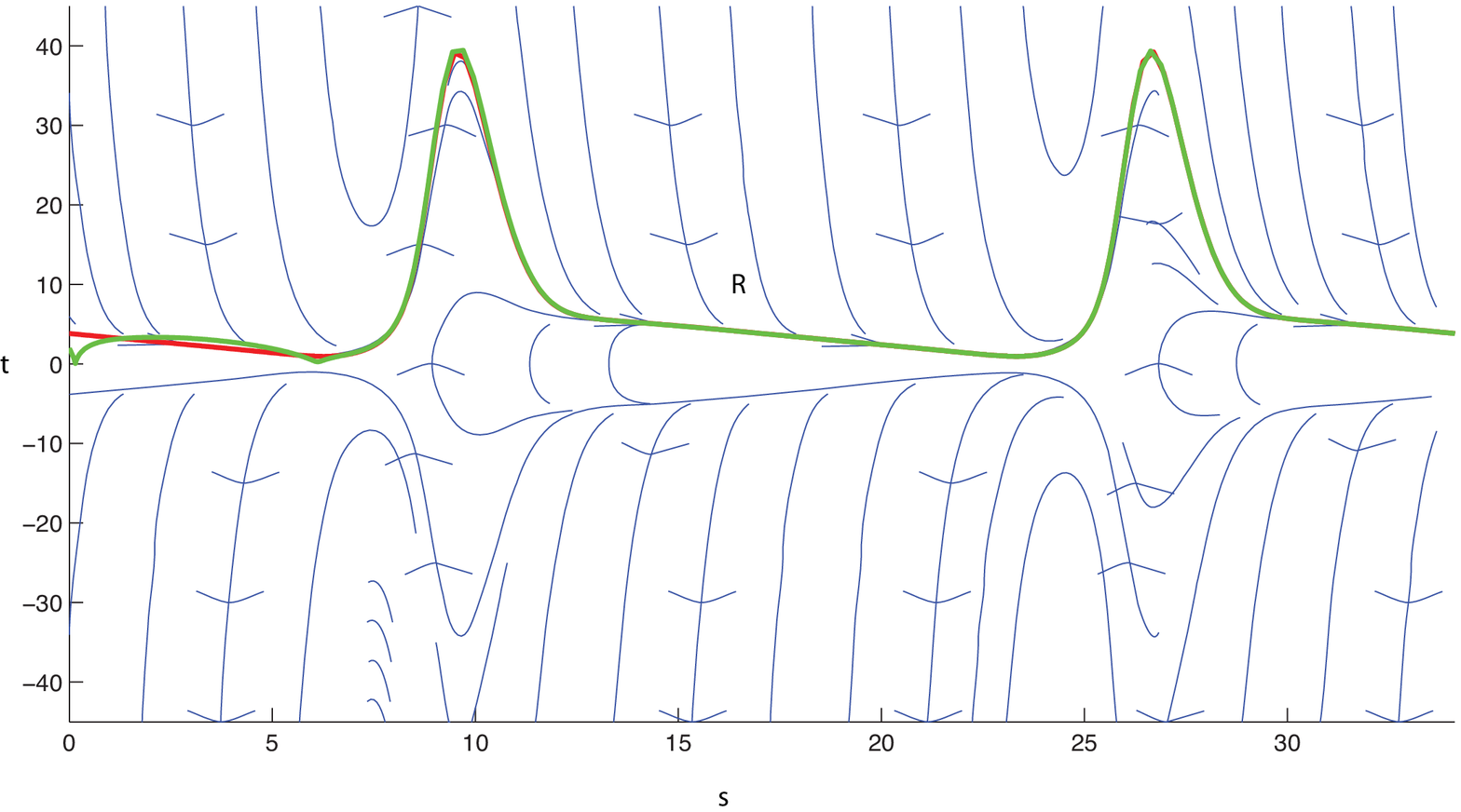}}
\caption{Simulation of the closed-loop system for the curve considered in
  example~\ref{example_non_convex}. The solution converges to the submanifold $\cR \subset \Gamma$.}
\label{fig:non_convex:closed_loop1}
\end{figure}
}\end{example}

\section*{Acknowledgements} 
We are grateful to Alireza Mohammadi for his helpful comments
on this paper. 

\bibliographystyle{elsarticle-num}\bibliography{citazioni}

\end{document}